\documentclass[12pt,reqno]{amsart}
\usepackage[english]{babel}
\usepackage{amsmath,amsfonts,amssymb,amsthm,amscd,latexsym}
\usepackage[T1]{fontenc}
\usepackage[utf8]{inputenc}

\usepackage{color}
\usepackage{cite}
\usepackage{multirow}
\usepackage{cite}

 \usepackage{hyperref}                  % Reference

\newtheorem{theorem}{Theorem}[section]
\newtheorem{lemma}[theorem]{Lemma}
\newtheorem{proposition}[theorem]{Proposition}

\theoremstyle{definition}

\theoremstyle{remark}

\numberwithin{equation}{section}

\begin{document}

\setcounter{page}{1}

\title[Local and 2-local derivations of   locally finite simple Lie algebras]{Local and 2-local derivations of  locally finite simple Lie algebras}

\author[ Ayupov Sh.A., Kudaybergenov K.K., Yusupov B.B. ]{Shavkat Ayupov$^{1,3}$, Karimbergen Kudaybergenov$^2$, Bakhtiyor Yusupov$^3$}
\address{$^1$ V.I.Romanovskiy Institute of Mathematics\\
  Uzbekistan Academy of Sciences, 81 \\ Mirzo Ulughbek street, 100170  \\
  Tashkent,   Uzbekistan}
 \address{$^2$ Department of Mathematics, Karakalpak State University, 1, Academician  Ch.~Abdirov street, 230113, Nukus,  Uzbekistan}
\address{$^3$ National University of Uzbekistan, 4, University street, 100174, Tashkent, Uzbekistan}
\email{\textcolor[rgb]{0.00,0.00,0.84}{sh$_{-}$ayupov@mail.ru, shavkat.ayupov@mathinst.uz}}

\email{\textcolor[rgb]{0.00,0.00,0.84}{karim2006@mail.ru}}
\email{\textcolor[rgb]{0.00,0.00,0.84}{baxtiyor\_yusupov\_93@mail.ru}}
\maketitle

\begin{abstract} In the present paper we study local and 2-local derivations of locally finite split simple Lie algebras. Namely, we show that every local and 2-local derivation on such   Lie algebra is a derivation.

\end{abstract}
\\


\section{Introduction}

The notion of local derivation were first introduced in 1990 by R.V.Kadison \cite{Kadison} and D.R.Larson, A.R.Sourour \cite{Larson}.
A linear operator $\Delta$ on an algebra $\mathcal{A}$ is called a \textit{local derivation} if given any $x\in\mathcal{A}$ there exists a derivation $D_x$(depending on $x$) such that $\Delta(x)=D_x(x).$ The main problems concerning this notion are to find conditions under which local derivations become derivations and to present examples of algebras with local derivations that are not derivations.
R.V.Kadison proved that each continuous local
derivation of a von Neumann algebra $M$ into a dual Banach $M$-bimodule is a derivation.

Investigation of local and 2-local derivations on finite dimensional Lie algebras were initiated in papers \cite{Ayupov7, AyuKudRak}.  In \cite{Ayupov7} the first two authors 
have proved that every local
derivation on semi-simple Lie algebras is a derivation and gave examples of nilpotent
finite-dimensional Lie algebras with local derivations which are not derivations. In \cite{Ayupov6} local derivations of solvable Lie algebras are investigated and it is shown that any local derivation of solvable Lie algebra with model nilradical
is a derivation.

In 1997, P.\v{S}emrl \cite{Sem} introduced the notion of  2-local derivations and
2-local automorphisms on algebras. Namely, a map \(\nabla : \mathcal{A} \to
\mathcal{A}\) (not necessarily linear) on an algebra \(\mathcal{A}\) is called a \textit{2-local
derivation}, if for every pair of elements \(x,y \in  \mathcal{A}\) there exists a
derivation \(D_{x,y} : \mathcal{A} \to \mathcal{A}\) such that
\(D_{x,y} (x) = \nabla(x)\) and \(D_{x,y}(y) = \nabla(y).\) The notion of 2-local automorphism is given in a similar way.  For a
given algebra \(\mathcal{A}\), the main problem concerning
these notions is to prove that they automatically become a
derivation (respectively, an automorphism) or to give examples of
local and 2-local derivations or automorphisms of \(\mathcal{A},\)
which are not derivations or automorphisms, respectively.

Solution
of such problems for finite-dimensional Lie algebras over
algebraically closed field of zero characteristic were obtained in
\cite{AyuKud, AyuKudRak, ChenWang}. Namely, in
\cite{AyuKudRak} it was proved that every 2-local derivation on a
semi-simple Lie algebra \(\mathcal{L}\) is a derivation and that
each finite-dimensional nilpotent Lie algebra, with dimension
larger than two admits 2-local derivation which is not a
derivation.  
Concerning 2-local
automorphism, Z.Chen and D.Wang in \cite{ChenWang} proved that if \(\mathcal{L}\) is
a simple Lie algebra of type $A_{l},D_{l}$ or $E_{k}, (k = 6, 7,
8)$ over an algebraically closed field of characteristic zero
then every 2-local automorphism of \(\mathcal{L}\) is an automorphism. Finally,
in \cite{AyuKud} it was  proved that every 2-local automorphism of a
finite-dimensional semi-simple Lie algebra over an algebraically
closed field of characteristic zero is an automorphism. Moreover,
have shown also that every nilpotent Lie algebra with finite
dimension larger than two admits 2-local automorphisms which is
not an automorphism.

In \cite{Ayupov8, AyuYus} the authors studied 2-local derivations of infinite-dimensional Lie algebras over a field of characteristic zero and proved  that all 2-local derivations of the Witt algebra as well as of the positive Witt algebra and the classical one-sided Witt algebra are (global) derivations and every 2-local derivation on Virasoro algebras is a derivation. In \cite{AyuKudYus} we have proved that every 2-local derivation on the generalized Witt algebra $W_n(\mathbb{F})$ over the vector space $\mathbb{F}^n$  is a derivation.
In \cite{YangKai} Y.Chen, K.Zhao and Y.Zhao studied local derivations on generalized Witt algebras. They proved that every local derivation on Witt algebras is a derivation and that every local derivation on a centerless generalized Virasoro algebra of higher rank is a derivation.

In the present paper we study local and 2-local derivations of  locally finite split simple Lie algebras.

\section{Preliminaries}

In this section we give some necessary definitions and preliminary results (for details see \cite{Neeb2005, Neeb2001}).

 A Lie algebra $\mathfrak{g}$ over a field $\mathbb{F}$ is a vector space over $\mathbb{F}$ with a bilinear
mapping $\mathfrak{g}\times\mathfrak{g}\rightarrow\mathfrak{g}$ denoted $(x,y)\mapsto[x,y]$ and called the bracket of $\mathfrak{g}$ and satisfying:
$$
[x,x] =0 ,\ \ \ \  \forall x\in\mathfrak{g},
$$
$$
[[x,y],z]+[[y,z],x]+[[z,x],y]=0, \forall x,y,z\in\mathfrak{g}.
$$

A Lie algebra $\mathfrak{g}$ is said
to be \textit{solvable} if $\mathfrak{g}^{(k)}=\{0\}$ for some
integer $k,$ where $\mathfrak{g}^{(0)}=\mathfrak{g},$
$\mathfrak{g}^{(k)}=\Big[\mathfrak{g}^{(k-1)}, \mathfrak{g}^{(k-1)}\Big],\,
k\geq1.$ Any Lie algebra $\mathfrak{g}$ contains a unique maximal
solvable ideal, called the radical of $\mathfrak{g}$ and denoted by
$\mbox{Rad} \mathfrak{g}.$ A non trivial Lie algebra $\mathfrak{g}$
is called \textit{semisimple} if $\mbox{Rad} \mathfrak{g}=0.$ That
is equivalent to requiring that $\mathfrak{g}$ have no nonzero
abelian ideals. A Lie algebra $\mathfrak{g}$ is simple, if it has no non-trivial ideals and is not abelian.

We say that a Lie algebra $\mathfrak{g}$ has a \textit{root decomposition} with respect to an abelian subalgebra $\mathfrak{h},$ if
$$
\mathfrak{g}=\mathfrak{h}\oplus\bigoplus\limits_{\alpha\in \mathfrak{R}}\mathfrak{g}_{\alpha},
$$
where $\mathfrak{g}_{\alpha}=\Big\{x\in\mathfrak{g}: [h,x]=\alpha(h)x,\,\, \forall h\in\mathfrak{h}\Big\}$
and $\mathfrak{R}=\Big\{\alpha\in\mathfrak{h}^*\backslash0:\mathfrak{g}_{\alpha}\neq\{0\}\Big\}$
 is the corresponding root system and $\mathfrak{h}^*$ is the space of all linear functionals on $\mathfrak{h}.$
In this case, $\mathfrak{h}$ is called \textit{splitting Cartan subalgebra} of $\mathfrak{g},$ and $\mathfrak{g}$
respectively the pair $(\mathfrak{g},\mathfrak{h})$ is called \textit{split} Lie algebra.

Suppose that  $\mathfrak{g}$ is a Lie algebra over $\mathbb{F}$ which is a directed
union of  finite-dimensional simple Lie algebras, that is,  $\mathfrak{g}=\lim\limits_{\longrightarrow}\mathfrak{g}_\alpha$
is the direct limit of a family $(\mathfrak{g}_\alpha)_{\alpha\in A}$ of finite-dimensional simple Lie algebras $\mathfrak{g}_\alpha$
which are subalgebras of $\mathfrak{g}$ and the directed order $\leq$ on the index set $A$ is given
by $\alpha\leq \beta$ if $\mathfrak{g}_\alpha\leq \mathfrak{g}_\beta.$
A Lie algebra $\mathfrak{g}$  of this form  is said to be  locally finite simple Lie algebra.

 Now following \cite{Neeb2005} we give a description of locally
finite split simple Lie algebras.

For a set $\mathfrak{J}$  we denote by $M_\mathfrak{J}(\mathbb{F})=\mathbb{F}^{\mathfrak{J}\times\mathfrak{J}}$
 the set of all 
 $\mathfrak{J}\times\mathfrak{J}$-matrices with entries in $\mathbb{F}.$
 Let   $M_\mathfrak{J}(\mathbb{F})_{rc-fin}\subseteq M_\mathfrak{J}(\mathbb{F})$ be 
the set of all $\mathfrak{J}\times\mathfrak{J}$-matrices
  with at most finitely many non-zero entries in each row and each column, and let $\mathfrak{gl}_{\mathfrak{J}}(\mathbb{F})$  denote the subspace
   consisting of all matrices with at most finitely many non-zero entries.
The matrix product $xy$  is defined if at least one factor is in $\mathfrak{gl}_{\mathfrak{J}}(\mathbb{F})$ and the other
is in $M_\mathfrak{J}(\mathbb{F}).$
In particular, $\mathfrak{gl}_{\mathfrak{J}}(\mathbb{F})$ thus inherits the structure of locally finite Lie algebra via $[x,y]:=xy-yx$ and
\begin{equation*}\begin{split}
\mathfrak{sl}_{\mathfrak{J}}(\mathbb{F})=\left\{x\in \mathfrak{gl}_{\mathfrak{J}}(\mathbb{F}): tr(x)=0\right\}
\end{split}\end{equation*}
is a simple Lie algebra.

For $i,j\in \mathfrak{J}$ denote  by $e_{i,j}$ a matrix unit defined as 
$$
e_{i,j}:\mathfrak{J}\times\mathfrak{J}\rightarrow\mathbb{F},\ \ (k,s)\longmapsto\delta_{ik}\delta_{sj},
$$
where $\delta_{i,j}$ is the Kronecker symbol.

Set  $2\mathfrak{J}:=\mathfrak{J}\,\dot{\cup}-\mathfrak{J},$ where $-\mathfrak{J}$ denotes a copy of $\mathfrak{J}$ whose elements are denoted by $-i\,(i\in\mathfrak{J})$ and consider the $2\mathfrak{J}\times2\mathfrak{J}$-matrices
\begin{equation*}
q_1=\sum\limits_{i\in\mathfrak{J}}e_{i,-i}+e_{-i,i}\ \ \text{and} \ \ q_2=\sum\limits_{i\in\mathfrak{J}}e_{i,-i}-e_{-i,i}.
\end{equation*}
Set
\begin{equation*}
\mathfrak{o}_{\mathfrak{J},\mathfrak{J}}(\mathbb{F})=\left\{x\in\mathfrak{gl}_{2\mathfrak{J}}(\mathbb{F}):\ x^{\top} q_1+q_1x=0\right\}
\end{equation*}
and
\begin{equation*}
\mathfrak{sp}_{\mathfrak{J}}(\mathbb{F})=\left\{x\in\mathfrak{gl}_{2\mathfrak{J}}(\mathbb{F}):\ x^{\top} q_2+q_2x=0\right\}.
\end{equation*}

By \cite[Theorem IV.6]{Neeb2001} 
every infinite dimensional locally finite split simple Lie algebra is isomorphic
to one of the Lie algebras $\mathfrak{sl}_{\mathfrak{J}}(\mathbb{F}), \mathfrak{o}_{\mathfrak{J},\mathfrak{J}}(\mathbb{F}), \mathfrak{sp}_{\mathfrak{J}}(\mathbb{F}),$ where $\mathfrak{J}$ is an infinite set with
$\textrm{card}\mathfrak{J} = \dim\mathfrak{g}.$

In the next section we shall use the following description of algebras of derivations of locally finite simple Lie algebras \cite[Theorem I.3]{Neeb2005}:
\begin{equation*}
\begin{split}
der\left(\mathfrak{sl}_{\mathfrak{J}}(\mathbb{F}\right) & \cong M_\mathfrak{J}(\mathbb{F})_{rc-fin}/\mathbb{F}\mathbf{1}\\
der\left(\mathfrak{o}_{\mathfrak{J},\mathfrak{J}}(\mathbb{F})\right) & \cong \left\{x\in M_{\mathfrak{J}}(\mathbb{F})_{rc-fin}:x^\top q_1+q_1x=0\right\}\\
der\left(\mathfrak{sp}_\mathfrak{J}(\mathbb(F))\right) & \cong \left\{x\in M_{\mathfrak{J}}(\mathbb{F})_{rc-fin}:x^\top q_2+q_2x=0\right\},
\end{split}
\end{equation*}
where  $\mathbf{1}=(\delta_{ij})$ is the indentity matrix in
$M_\mathfrak{J}(\mathbb{F}).$ In particular, any derivation $D$ on $\mathfrak{sl}_{\mathfrak{J}}(\mathbb{F})$ represented as
\begin{equation}\label{dersplit}
D(x)=[a,x],\,\, x\in \mathfrak{sl}_{\mathfrak{J}}(\mathbb{F}), 
\end{equation}
where $a\in M_\mathfrak{J}(\mathbb{F})_{rc-fin}.$ Further, in the cases of algebras $\mathfrak{o}_{\mathfrak{J},\mathfrak{J}}(\mathbb{F})$ and  $\mathfrak{sp}_{\mathfrak{J}}(\mathbb{F})$ an element $a\in M_\mathfrak{J}(\mathbb{F})_{rc-fin}$ need to satisfy conditions  
$a^\top q_1+q_1 a=0$ and $a^\top q_2+q_2 a=0,$ respectively.

\section{Main results}

\subsection{Local derivation  on the locally finite split simple Lie algebras}

\
\medskip

The main result of this subsection is given as follows.
\begin{theorem}\label{th21}
Let $\mathfrak{g}$ be a locally
finite split simple Lie algebras over a field of characteristic zero. Then any
local derivation on $\mathfrak{g}$ is a derivation.
\end{theorem}

For the proof  we need several Lemmata and from now on
$\mathfrak{g}$ is an one of  the algebras
 $\mathfrak{sl}_{\mathfrak{J}}(\mathbb{F}), \mathfrak{o}_{\mathfrak{J},\mathfrak{J}}(\mathbb{F}), \mathfrak{sp}_{\mathfrak{J}}(\mathbb{F})$ (see the end of the
previous Section).

Any  $x\in M_\mathfrak{J}(\mathbb{F})_{rc-fin}$ can be uniquely represented as
$$
x=\sum\limits_{i,j\in \mathfrak{J}} x_{i,j}e_{i,j},
$$
where $x_{i,j}\in \mathbb{F}$ for all $i,j\in  \mathfrak{J}.$

For a  subset $\mathfrak{I} \subset \mathfrak{J}$ we shall identify the algebra $M_\mathfrak{I}(\mathbb{F})_{rc-fin}$ 
with the subalgebra in $M_\mathfrak{J}(\mathbb{F})_{rc-fin},$ consisting of elements of the form
$
x=\sum\limits_{i,j\in \mathfrak{I}} x_{i,j}e_{i,j},
$
where $x_{i,j}\in \mathbb{F}$ for all $i,j\in  \mathfrak{I}.$ 
Further, for a finite subset $\mathfrak{I} \subset \mathfrak{J}$
we define  a projection mapping $\pi_\mathfrak{I}:M_\mathfrak{J}(\mathbb{F})_{rc-fin}\rightarrow M_\mathfrak{I}(\mathbb{F})_{rc-fin}$ as follows
\begin{equation*}
\pi_\mathfrak{I}(x)=\sum\limits_{i,j\in \mathfrak{I}} x_{i,j}e_{i,j},
\end{equation*}
where $x=\sum\limits_{i,j\in \mathfrak{J}} x_{i,j}e_{i,j}.$

\begin{lemma}\label{123} 
$$
\pi_\mathfrak{I}\left([x,y]\right)=\left[\pi_\mathfrak{I}(x), \pi_\mathfrak{I}(y)\right]
$$
for all $x\in M_\mathfrak{J}(\mathbb{F})_{rc-fin}$ and $y\in M_\mathfrak{I}(\mathbb{F})_{rc-fin}.$
\end{lemma}

 \begin{proof}
Note that each matrix $x\in M_\mathfrak{J}(\mathbb{F})_{rc-fin}$ is represented as $x=x_{\mathfrak{I},\mathfrak{I}}+x_{\mathfrak{I},\mathfrak{K}}+x_{\mathfrak{K},\mathfrak{I}}+x_{\mathfrak{K}, \mathfrak{K}},$
where  $x_{\mathfrak{I},\mathfrak{I}}=\sum\limits_{i,j\in \mathfrak{I}} x_{i,j}e_{i,j},$ $x_{\mathfrak{I},\mathfrak{K}}=\sum\limits_{i\in \mathfrak{I}, j\in \mathfrak{K}} x_{i,j}e_{i,j},$
$x_{\mathfrak{K},\mathfrak{I}}=\sum\limits_{i\in \mathfrak{K}, j\in \mathfrak{I}} x_{i,j}e_{i,j},$
$x_{\mathfrak{K},\mathfrak{K}}=\sum\limits_{i,j\in \mathfrak{K}} x_{i,j}e_{i,j}$ and $\mathfrak{K}=\mathfrak{J}\setminus\mathfrak{I}.$
Take the matrices $x=x_{\mathfrak{I},\mathfrak{I}}+x_{\mathfrak{I},\mathfrak{K}}+x_{\mathfrak{K},\mathfrak{I}}+x_{\mathfrak{K}, \mathfrak{K}}\in M_\mathfrak{J}(\mathbb{F})_{rc-fin}$
and $y=y_{\mathfrak{I},\mathfrak{I}}\in M_\mathfrak{I}(\mathbb{F})_{rc-fin}.$
Then
 \begin{equation*}\begin{split}
\pi_\mathfrak{I}\left([x,y]\right)&=\pi_\mathfrak{I}\left(\left[x_{\mathfrak{I},\mathfrak{I}}+x_{\mathfrak{I},\mathfrak{K}}+x_{\mathfrak{K},\mathfrak{I}}+x_{\mathfrak{K}, \mathfrak{K}},
y_{\mathfrak{I},\mathfrak{I}}\right]\right)\\
&=
\pi_\mathfrak{I}\left([x_{\mathfrak{I},\mathfrak{I}}, y_{\mathfrak{I},\mathfrak{I}}]\right)+\pi_\mathfrak{I}\left(\left[x_{\mathfrak{I},\mathfrak{K}}+x_{\mathfrak{K},\mathfrak{I}}+x_{\mathfrak{K}, \mathfrak{K}},
y_{\mathfrak{I},\mathfrak{I}}\right]\right)\\
&=[x_{\mathfrak{I},\mathfrak{I}}, y_{\mathfrak{I},\mathfrak{I}}]
=
\left[\pi_\mathfrak{I}(x), \pi_\mathfrak{I}(y)\right].
  \end{split}\end{equation*}
 \end{proof}

For a  subset $\mathfrak{I} \subset \mathfrak{J}$ denote by $\mathfrak{g}_\mathfrak{I}$ the subalgebra in $\mathfrak{g}$ consisting of elements of the form
$
x=\sum\limits_{i,j\in \mathfrak{I}} x_{i,j}e_{i,j}\in \mathfrak{g},
$
where $x_{i,j}\in \mathbb{F}$ for all $i,j\in  \mathfrak{I}.$

It is clear that the restriction $\pi_\mathfrak{I}|_{\mathfrak{g}}$ of $\pi_\mathfrak{I}$ on $\mathfrak{g}$ maps $\mathfrak{g}$ onto $\mathfrak{g}_\mathfrak{I}.$

\begin{lemma}\label{resder}
Let $\Delta$ be a local derivation on $\mathfrak{g}.$ Then the mapping $\Delta_\mathfrak{I}$ on $\mathfrak{g}_\mathfrak{I}$ defined as
$$
\Delta_\mathfrak{I}(x)=\pi_\mathfrak{I}(\Delta(x)),\  x\in\mathfrak{g}_\mathfrak{I}
$$
is a local derivation for all finite subset $\mathfrak{I}$ in $\mathfrak{J}.$
\end{lemma}

\begin{proof} Let $x\in \mathfrak{g}_\mathfrak{I}$ be an arbitrary element. By \eqref{dersplit} there is an  element  $a_x\in M_\mathfrak{J}(\mathbb{F})_{rc-fin}$ such that $\left[a_x, \mathfrak{g}\right]\subseteq \mathfrak{g}$ and
$
\Delta(x)=[a_x,x].
$
Then by Lemma \ref{123}
\begin{equation*}
\Delta_\mathfrak{I}(x)=\pi_\mathfrak{I}(\Delta(x))=\pi_\mathfrak{I}([a_x,x])=
                         [\pi_\mathfrak{I}(a_x),\pi_\mathfrak{I}(x)]=[\pi_\mathfrak{I}(a_x),x].
\end{equation*}
Thus $\Delta_\mathfrak{I}$ is a local derivation. \end{proof}

{\it Proof of Theorem \ref{th21}.} Let  $\Delta$ be a local derivation $\mathfrak{g}$ and let $x\in \mathfrak{g}$ be an arbitrary element. Take a  finite subset $\mathfrak{I}$ in $\mathfrak{J}$ such that $x, y, \Delta(x), \Delta(y), \Delta([x,y])\in\mathfrak{g}_\mathfrak{I}.$
 By Lemma~\ref{resder}, $\Delta_\mathfrak{I}$ is a local derivation of $\mathfrak{g}_\mathfrak{I}.$ Since $\mathfrak{g}_\mathfrak{I}$ is a finite dimensional simple Lie algebra, by \cite[Theorem 3.1]{Ayupov7} $\Delta_\mathfrak{I}$ is a derivation. Hence,
\begin{equation*}\begin{split}
\Delta([x,y]) &=\Delta_\mathfrak{I}([x,y])=[\Delta_\mathfrak{I}(x),y]+[x,\Delta_\mathfrak{I}(y)]=[\Delta(x),y]+[x,\Delta(y)].
\end{split}
\end{equation*}
This means that $\Delta$ is a derivation.

\subsection{2-local derivations on the  locally finite split simple Lie algebras}

\

In this subsection we study 2-local derivations on the locally finite split simple Lie algebras.

Recall that a bilinear from $\kappa$ on $\mathfrak{g}$ is said to be  non degenerate, i.e. $\kappa(x, y)=0$ for
all $y\in \mathfrak{g}$ implies that $x=0.$

We shall use the following results from \cite{Neeb2005}.

\begin{proposition}
There exists a nondegenerate invariant symmetric bilinear
form $\kappa$ on $\mathfrak{g}.$
\end{proposition}

\begin{proposition}
	Every invariant symmetric bilinear form $\kappa$ on $\mathfrak{g}$ is invariant under all derivations of $\mathfrak{g}.$
\end{proposition}

The main results of this subsection is the following.

\begin{theorem}\label{th32}
Let $\mathfrak{g}$ be a locally finite split simple Lie algebras over a field of characteristic zero. Then any
2-local derivation on $\mathfrak{g}$ is a derivation.
\end{theorem}

\begin{proof}
 Let us first to show that $\nabla$ is  linear.
 
 Let $x,y,z\in \mathfrak{g}$ be arbitrary elements. Taking into account invariance of $\kappa$  we obtain
\begin{equation*}\begin{split}
\kappa(\nabla(x+y),z)&=\kappa(D_{x+y,z}(x+y),z)=-\kappa(x+y,D_{x+y,z}(z))\\
&=-\kappa(x+y,\nabla(z))=-\kappa(x,\nabla(z))-\kappa(y,\nabla(z))\\
&=-\kappa(x,D_{x,z}(z))-k(y,D_{y,z}(z))=\kappa(D_{x,z}(x),z)\\
&+\kappa(D_{y,z},z)=\kappa(\nabla(x),z)+\kappa(\nabla(y),z)\\
&=\kappa(\nabla(x)+\nabla(y),z),
\end{split}
\end{equation*}
i.e.
\begin{equation*}
\kappa(\nabla(x+y),z)=\kappa(\nabla(x)+\nabla(y),z)
\end{equation*}
Since $\kappa$  is non-degenerate the last equality implies that

\begin{equation*}
\nabla(x+y)=\nabla(x)+\nabla(y)\ \ \  \text{for}\  x,y\in\mathfrak{g}.
\end{equation*}
Further,
\begin{equation*}
\nabla(\lambda x)=D_{\lambda x,x}(\lambda x)= \lambda D_{\lambda x,x}(x)=\lambda\nabla(x).
\end{equation*}
Hence, $\nabla$ is linear, and therefore is a local derivation.  

Finally, by Theorem \ref{th21} a local derivation $\nabla$ is a derivation.
\end{proof}


\begin{thebibliography}{99}

\bibitem{AyuKud} Ayupov Sh.A., Kudaybergenov K.K., {\it 2-Local automorphisms on finite dimensional Lie algebras,} Linear Algebra and its
 Applications, \textbf{507}, 121--131 (2016).

 \bibitem{Ayupov7} Ayupov Sh.A., Kudaybergenov K.K., {\it Local derivation on finite dimensional Lie algebras,} Linear Algebra and its
 Applications, \textbf{493}, 381--398 (2016).

\bibitem{AyuKudRak}Ayupov Sh.A., Kudaybergenov K.K., Rakhimov I.S., {\it 2-Local derivations on finite-dimensional Lie algebras,} Linear Algebra and its Applications, \textbf{474}, 1--11 (2015).


\bibitem{AyuKudYus}Ayupov Sh.A., Kudaybergenov K.K., Yusupov B.B., {\it 2-Local derivations on generalized Witt algebras,} Linear and Multilinear Algebra, (2019) DOI: 10.1080/03081087.2019.1708846.

\bibitem{Ayupov6} Ayupov Sh.A.,  Khudoyberdiyev A.Kh., {\it Local derivations on solvable Lie algebras,} Linear and Multilinear Algebra,
DOI:10.1080/03081087.2019.1626336.


\bibitem{Ayupov8}  Ayupov Sh.A., Yusupov B.B., {\it 2-local derivations of infinite-dimensional Lie algebras,} Journal of Algebra and its Aplications,  
DOI:10.1142/S0219498820501005.

\bibitem{AyuYus}  Ayupov Sh.A., Yusupov B.B., {\it 2-local derivations on Virasoro algebras,} Bulletin of National University of Uzbekistan: Mathematics and Natural Sciences,  {\bf 2}:4 217--230 (2019).

\bibitem{YangKai} Chen Y., Zhao K.,  Zhao Y., {\it Local derivations on Witt algebras,} arXiv:1911.05.15v1 [math.RA] 12 Nov (2019).

\bibitem{ChenWang} Chen Z., Wang D., {\it 2-Local automorphisms of finite-dimensional simple Lie algebras,}  Linear Algebra and its Applications,  \textbf{486},  335--344 (2015).


\bibitem{Neeb2005} Nebb K.--H., {\it Derivations of locally simple Lie Algebras,} Journal of Lie Theory,  \textbf{15},  589-594 (2005).

\bibitem{Neeb2001} Neeb K.--H.,  Stumme N., {\it The classification of locally finite split simple Lie algebras,} J. Reine Angew. Math. \textbf{533}, 25–-53 (2001).

\bibitem{Sem} \v{S}emrl P., {\it Local automorphisms and derivations on $B(H),$} Proceedings of the American Mathematical Society, \textbf{125}, 2677--2680 (1997).


\bibitem{Kadison} Kadison R.V,  {\it Local derivations,}
Journal of Algebra,  \textbf{130},  494--509 (1990).

\bibitem{Larson} Larson D.R., Sourour A.R.,  {\it Local derivations and local automorphisms of $B(X)$,}
Proceedings of Symposia in Pure Mathematics, \textbf{51} Part 2, Provodence, Rhode Island, p. 187--194 (1990).

\end{thebibliography}
\end{document}